\documentclass[12pt,reqno]{amsart}
\usepackage{fullpage}
\allowdisplaybreaks
\usepackage{times}
\usepackage{colonequals}
\usepackage{amsmath,amssymb,amsthm,url}
\usepackage[utf8]{inputenc}
\usepackage[english]{babel}
\usepackage[alphabetic]{amsrefs}
\usepackage{bbm}
\usepackage{enumerate}
\usepackage{bm}
\usepackage{graphicx}
\usepackage{mathrsfs}
\usepackage[colorlinks=true, pdfstartview=FitH, linkcolor=blue, citecolor=blue, urlcolor=blue]{hyperref}
\newtheorem{thm}{Theorem}[section]
\newtheorem{lem}[thm]{Lemma}
\newtheorem{prop}[thm]{Proposition}
\newtheorem{cor}[thm]{Corollary}

\theoremstyle{definition}
\newtheorem{defn}[thm]{Definition}
\newtheorem{question}[thm]{Question}
\newtheorem{rem}[thm]{Remark}

\newtheorem{ex}[thm]{Example}

\newcommand{\PP}{\mathcal{P}}
\newcommand{\bP}{\mathbb{P}}

\newcommand{\F}{\mathbb{F}}

\newcommand{\nb}{\operatorname{nb}}

\AtBeginDocument{%
	\def\MR#1{}
}

\begin{document}

\title{Most plane curves over finite fields are not blocking}

\author{Shamil Asgarli}
\address{Department of Mathematics and Computer Science \\ Santa Clara University \\ 500 El Camino Real \\ USA 95053}
\email{sasgarli@scu.edu}

\author{Dragos Ghioca}
\address{Department of Mathematics \\ University of British Columbia \\ 1984 Mathematics Road \\ Canada V6T 1Z2}
\email{dghioca@math.ubc.ca}

\author{Chi Hoi Yip}
\address{Department of Mathematics \\ University of British Columbia \\ 1984 Mathematics Road \\ Canada V6T 1Z2}
\email{kyleyip@math.ubc.ca}


\maketitle

\begin{abstract}
A plane curve $C\subset\mathbb{P}^2$ of degree $d$ is called \emph{blocking} if every $\mathbb{F}_q$-line in the plane meets $C$ at some $\mathbb{F}_q$-point. We prove that the proportion of blocking curves among those of degree $d$ is $o(1)$ when $d\geq 2q-1$ and $q \to \infty$. We also show that the same conclusion holds for smooth curves under the somewhat weaker condition $d\geq 3p$ and $d, q \to \infty$. Moreover, the two events in which a random plane curve is smooth and respectively blocking are shown to be asymptotically independent. Extending a classical result on the number of $\mathbb{F}_q$-roots of random polynomials, we find that the limiting distribution of the number of $\mathbb{F}_q$-points in the intersection of a random plane curve and a fixed $\mathbb{F}_q$-line is Poisson with mean $1$. We also present an explicit formula for the proportion of blocking curves involving statistics on the number of $\mathbb{F}_q$-points contained in a union of $k$ lines for $k=1, 2, \ldots, q^2+q+1$.
\end{abstract}

\section{Introduction}\label{sect:intro}

We start with the definition of the main combinatorial structure in this paper. As usual, $p$ denotes a prime, $q$ denotes a power of $p$, and $\F_q$ denotes the finite field with $q$ elements. A set of points $B\subseteq \bP^2(\F_q)$ is a \emph{blocking set} if every $\F_q$-line meets $B$. The simplest example of a blocking set is a union of $q+1$ distinct $\F_q$-points on a line. A blocking set $B$ is \emph{trivial} if it contains all the $q+1$ points of a line and is \emph{nontrivial} if it is not a trivial blocking set. 

Given a projective plane curve $C\subset\mathbb{P}^2$ defined over $\F_q$, we say that $C$ is a \emph{blocking curve} if $C(\F_q)$ is a blocking set. Moreover, $C$ is \emph{nontrivially blocking} if $C(\F_q)$ is a nontrivial blocking set. In our previous paper \cite{AGY23}, we systematically studied blocking plane curves. In particular, we showed that an irreducible plane curve is not blocking whenever its degree is small compared to $q$ \cite{AGY23}*{Theorem 1.2}. We also constructed a few families of geometrically irreducible or smooth nontrivially blocking curves. Our motivation for studying blocking curves stems from the exciting interplay between finite geometry and algebraic geometry developed in recent years; see the surveys \cites{S97b, SS98} for more instances. We refer to \cite{AGY23} for more motivation to study smooth blocking curves, including potential applications in algebraic geometry.

The present paper aims to study blocking curves through the lens of arithmetic statistics. We prove several precise results in this direction. In particular, we show that the property of being blocking is rare whenever the degree is large compared to $q$. 

\begin{thm}\label{thm:most-curves-are-not-blocking}
There exists a function $\psi(x)$ with $\displaystyle\lim_{x\to\infty}\psi(x)=0$ such that the following holds. Let $\F_q$ be a fixed finite field, and $d\geq 2q-1$. Among all plane curves of degree $d$ defined over $\F_q$, the fraction of blocking curves is at most $\psi(q)$. In particular, most plane curves are not blocking.
\end{thm}

We have a similar result for smooth plane curves where the condition on the degree $d$ depends on the characteristic $p$ of the field $\F_q$. 

\begin{thm}\label{thm:most-smooth-curves-are-not-blocking}
Almost all smooth plane curves of degree $d$ over $\F_q$ are not blocking when $d, q\to\infty$ satisfying $d\geq \min \{2q,3p\}$.
\end{thm}

Note that Theorem~\ref{thm:most-smooth-curves-are-not-blocking} is somewhat weaker than Theorem~\ref{thm:most-curves-are-not-blocking} when $q$ is a prime number. However, when $q$ is a proper prime power (that is, of the form $p^n$ with $n\geq 2$), Theorem~\ref{thm:most-smooth-curves-are-not-blocking} provides a more refined result. For instance, we can deduce that the proportion of blocking curves of degree $d=\lfloor p\log q \rfloor$ over $\F_q$ tends to $0$ as $q\to\infty$.

Given a plane curve $C$ defined over $\mathbb{F}_q$, we say that a line $L$ is \emph{skew} to $C$ if $L\cap C$ has no $\mathbb{F}_q$-points. Using this language, a plane curve $C$ is not blocking if and only if $C$ admits a skew $\mathbb{F}_q$-line. Not only do we show that most curves are \emph{not} blocking, but we also expect that a plane curve of degree $d\geq q$ has many skew lines. More precisely, we prove the following effective result.

\begin{thm}\label{thm:expected-skew-lines}
Let $C\subset\mathbb{P}^2$ be a plane curve of degree $d$ defined over $\F_q$ satisfying $d\geq q$. Then the expected number of skew $\mathbb{F}_q$-lines for $C$ is
$$
\frac{q^2}{e} - \frac{q}{2e} -\frac{5}{24e}+O\bigg(\frac{1}{q}\bigg).
$$
In other words, we expect roughly 36.7879\% of all lines defined over $\mathbb{F}_q$ to be skew to $C$.
\end{thm}

More generally, Poisson distribution can be used to predict the number of skew lines to a random plane curve; see Section~\ref{subsect:poisson}. As we will explain later in Remark~\ref{rmk:entin}, there is also a connection between Theorem~\ref{thm:expected-skew-lines} and Entin's work on the Chebotarev density theorem \cite{E21}.

Let $S_d$ be the set of homogeneous polynomials $F(X,Y,Z)$ of degree
$d$ over $\F_q$ including $0 \in S_d$ in order to make $S_d$ into a vector space, and let $S_d^{\rm ns} \subseteq S_d$ be the subset
of polynomials corresponding to smooth (\textbf{ns} stands for non-singular) curves 
$C_F: F(X,Y,Z)=0$. 
Define
\begin{equation*}
\nb(q) = \lim_{d\to\infty} \frac{\# \{F\in S_d : C_F \text{ is not blocking}\}}{\# S_d}.
\end{equation*}

We find an explicit lower and upper bound for $\nb(q)$.

\begin{thm}\label{thm:lower-upper-bounds}
Given a finite field $\F_q$, define 
$$
\lambda_q(x) = x \left(1-\left(1-x^q\right)^{q+1}\right) + x \sum_{j=1}^q  \left(1-x\right)^j \left(1-\left(1-x^q\right)^{q}\right).
$$
We have
$$
\lambda_q \left(1-\frac{1}{q}\right) \leq \nb(q) \leq 1 - \lambda_q \left(\frac{1}{q}\right) .
$$
\end{thm}

In fact, the conclusion of Theorem~\ref{thm:lower-upper-bounds} holds even if $\nb(q)$ is replaced with the density of non-blocking curves of fixed degree $d$ as long as $d \geq 2q-1$; see Proposition~\ref{proportion-of-curves}.

The sequence $\lambda_q\left(1-\frac{1}{q}\right)$ is increasing in $q$ and tends to $1$ as $q\to\infty$, thus $\lim_{q\to\infty} \nb(q)=1$. In particular, Theorem~\ref{thm:lower-upper-bounds} implies
$$
\nb(q) \geq \lambda_{11}\left(1-\frac{1}{11}\right) >0.994
$$
for each $q\geq 11$. We conclude that a random plane curve of a large degree over $\F_q$ does \emph{not} give rise to a blocking set with probability at least $0.99$, provided that $q\geq 11$. 

Similarly, we can define the analogous constant for smooth curves:
\begin{equation*}
\nb^{\mathrm{ns}}(q) = \lim_{d\to\infty} \frac{\# \{F\in S^{\mathrm{ns}}_d : C_F \text{ is not blocking}\}}{\# S^{\mathrm{ns}}_d}.
\end{equation*}
Theorem~\ref{thm:most-curves-are-not-blocking} and Theorem~\ref{thm:most-smooth-curves-are-not-blocking} provide effective estimates on the above two densities. In particular, we have $\lim_{q\to\infty} \nb(q)=1$ and $\lim_{q\to\infty} \nb^{\mathrm{ns}}(q)=1$, that is, almost all curves are not blocking, and the statement remains true when we restrict to smooth curves. It follows immediately that $\lim_{q\to\infty} \frac{\nb^{\mathrm{ns}}(q)}{\nb(q)}=1$. However, it is not clear whether the following limit 
$$
\lim_{q\to\infty} \frac{1-\nb^{\mathrm{ns}}(q)}{1-\nb(q)}
$$
exists and has any meaningful value. We will show that the limit above also approaches $1$.

\begin{thm}\label{thm:smooth-blocking-independence}
We have
$$
\lim_{q\to\infty} \frac{1-\nb^{\mathrm{ns}}(q)}{1-\nb(q)}=1.
$$
In particular, the events of being smooth and being blocking are asymptotically independent.
\end{thm}

It is necessary to take the limit as $q\to\infty$ in the previous theorem. Indeed, as we will see in  Theorem~\ref{thm:independence-refined}, the value of the fraction $\frac{1-\nb^{\mathrm{ns}}(q)}{1-\nb(q)}$ is strictly less than $1$ when $q$ is large. Similar techniques may be applied to combinatorial properties
of plane curves over $\F_q$ other than blocking lines. See remarks in Section~\ref{sect:independence-smooth-blocking}.

\subsection*{Related work} The study of combinatorial arrangements arising from plane curves dates back to 1950s. For example, Segre~\cite{S62} and Lombardo-Radice~\cite{LR56} constructed complete $(k, d)$-arcs in the plane from curves of degree $d$; more recently, Bartoli, Giulietti, and Zini \cite{BGZ16} constructed such arcs from curves of degree $d+1$. In these papers, $d$ is relatively small. Our focus is elucidating the interaction between arithmetic statistics and finite geometry where $d$ can be large. We briefly mention a few recent related works in this area. Let $C$ be a geometrically irreducible plane curve over $\F_q$ with degree $d$. Fix an integer $0\leq k \leq d$. Makhul, Schicho, and Gallet \cite{MSG20} studied the probability for a random line in $\bP^2(\F_{q^N})$ to intersect $C$ with exactly $k$ points, as $N \to \infty$. Yet another research avenue is arithmetic statistics of plane curves satisfying (or not satisfying) Bertini's theorem. Recently, Asgarli and Freidin \cite{AF21} computed upper and lower bounds for the density of plane curves over $\F_q$ that admit no transverse line over $\F_q$ in the setting when $d\to\infty$ (and $q$ is fixed). 

\subsection*{Outline of the paper} In Section~\ref{sect:prelim}, we review relevant background material from finite geometry and algebraic geometry. We also develop some of the key ingredients necessary for the later sections. Section~\ref{sect:skew-lines} highlights a linear independence argument, namely Proposition~\ref{prop:independence}, with application to the distribution of point counts and skew lines to plane curves. In Section~\ref{sect:proportion-general-case}, we prove our main Theorem~\ref{thm:most-curves-are-not-blocking}. This is achieved by proving Theorem~\ref{thm:lower-upper-bounds}, which gives lower and upper bounds on $\nb(q)$. We apply the existing machinery to deduce Theorem~\ref{thm:expected-skew-lines}. In Section~\ref{sect:proportion-smooth-case}, we focus on smooth curves and prove Theorem~\ref{thm:most-smooth-curves-are-not-blocking}, and in Section~\ref{sect:independence-smooth-blocking}, we prove Theorem~\ref{thm:smooth-blocking-independence}. Finally, Section~\ref{sect:proportion-small-fields} is devoted to finding the exact natural density of blocking plane curves over $\F_q$ when $q=2, 3, 4$. 

\section{Preliminaries}\label{sect:prelim}

\subsection{Collection of known results on blocking sets}

Understanding the possible sizes of blocking sets is an important problem in finite geometry. This topic has been studied extensively, and the following theorem summarizes some of the best-known lower bounds on the size of a nontrivial blocking set. Recall that if $q$ is a square, a {\em Baer subplane} in $\bP^2(\F_q)$ is a subplane with size $q+\sqrt{q}+1$. It is known that a Baer subplane is a blocking set. As a result, the lower bound $q+\sqrt{q}+1$ on $B$ is achieved by Baer subplanes, and in fact, only by Baer subplanes \cites{B70, B71}. Moreover, the lower bound can be improved significantly if $B$ avoids any Baer subplane \cite{BSS99}. 

\begin{thm}\label{thm: lb}
Let $q=p^r$. Let $c_2=c_3=2^{-1/3}$ and $c_p=1$ for $p>3$. Let $B$ be a nontrivial blocking set in $\bP^2(\F_q)$. Then the following lower bounds on $|B|$ hold:
\begin{itemize}
    \item (Bruen \cites{B70,B71}) $|B| \geq q+\sqrt{q}+1$. Moreover, the equality holds if and only if $q$ is a square and $B$ is a Baer subplane.
    \item (Blokhuis \cite{B94}) If $r=1$ and $p>2$, then $|B| \geq 3(q+1)/2$, where the lower bound is sharp by considering a projective triangle.
    \item (Sz\H{o}nyi \cite{S97}*{Theorem 5.7}) If $p>2$, $r=2$, and $B$ does not contain a Baer subplane, then $|B| \geq 3(q+1)/2$, where the lower bound is sharp by considering a projective triangle.
    \item (Blokhuis, Storme,  and Sz\H{o}nyi \cite{BSS99}*{Theorem 1.1}) If $r$ is odd, then $|B|\geq q+1+c_pq^{2/3}$.
    \item (Blokhuis, Storme,  and Sz\H{o}nyi \cite{BSS99}*{Theorem 1.1}) If $r$ is even, $q>16$ and $B$ does not contain a Baer subplane. Then $|B|\geq q+1+c_pq^{2/3}$.
\end{itemize}
\end{thm}

In Section~\ref{sect:proportion-smooth-case}, we will use the exponent $2/3$ in \cite{BSS99}*{Theorem 1.1} in a crucial way. 

\subsection{Statistics on the number of rational points of smooth curves}
Let $S_d$ be the set of homogeneous polynomials $F(X,Y,Z)$ of degree
$d$ over $\F_q$, and let $S_d^{\rm ns} \subseteq S_d$ be the subset
of polynomials corresponding to smooth curves
$C_F: F(X,Y,Z)=0$. The celebrated Hasse-Weil bound asserts that
$$
q+1-(d-1)(d-2)\sqrt{q} \leq \#C_F(\F_q) \leq q+1+(d-1)(d-2)\sqrt{q}
$$
for each $F \in S_d^{\rm ns}$. Note that the Hasse-Weil bound already eliminates the possibility of a low degree curve being blocking given Theorem~\ref{thm: lb}; we refer to \cite{AGY23}*{Section 3} for stronger statements. However, in this paper, we mainly work on the case when the degree is large, so relying on Hasse-Weil bound alone is insufficient.

From a probabilistic point of view, one anticipates that the expectation of $\#C_F(\F_q)$ is $q+1$, and the variance of $\#C_F(\F_q)$ is small. Note that Hasse-Weil bound gives the window for the number of $\F_q$ points where $q+1$ is the center of the interval. It is reasonable to expect that half of the curves should lie on one side of the interval. Note that this heuristic immediately implies that at least half of smooth curves are not nontrivially blocking in view of Theorem~\ref{thm: lb}. Bucur, David, Feigon, and Lal\'in \cite{BDFL10} proved a much stronger statement: when the degree $d$ is sufficiently large compared to $q$, the number of rational points tends to the standard normal distribution up to some normalization.

Let $X_1, \dots, X_{q^2+q+1}$ be $q^2+q+1$ independent and identically distributed (henceforth shortened to i.i.d.) random variables taking the
value $1$ with probability $(q+1)/(q^2+q+1)$ and the value $0$ with probability $q^2/(q^2+q+1).$ The key ingredient in their proof is the following random model: $\#C_F(\F_q) \approx X_1+X_2+\cdots+X_{q^2+q+1}$. In other words, the random model predicts that a random smooth curve will pass through each particular point with probability $(q+1)/(q^2+q+1)$, and all such events are independent. By estimating the corresponding moments, they showed that the random model above portrays the truth with high accuracy.

\begin{thm}[{\cite{BDFL10}*{Theorem 1.3}}]\label{moment}
Let $k$ be a positive integer, and let
$$
M_k(q,d) = \frac{1}{\# {S}_d^{\rm ns}} \sum_{F \in {S}_d^{\rm ns}} \left(
\frac{\# C_F(\F_q) - (q+1)}{\sqrt{q+1}} \right)^k.
$$
Then,
\begin{eqnarray*}
M_k(q,d) &=& {\mathbb{E}} \left( \left( \frac{1}{\sqrt{q+1}} \left( \sum_{i=1}^{q^2+q+1} X_i - (q+1) \right) \right)^k \right) \\
&\times & \left(1+ O\left( q^{\min(k, q^2+q+1)} \left( q^{-k}d^{-1/3} + (d-1)^2 q^{-\min\left(\left\lfloor \frac{d}{p}\right\rfloor+1, \frac{d}{3} \right)} + d q^{-\left\lfloor\frac{d-1}{p}\right\rfloor-1}\right)\right) \right).
\end{eqnarray*}
\end{thm}

The following corollary will be helpful in the proof of Theorem~\ref{thm:most-smooth-curves-are-not-blocking}.

\begin{cor}\label{second_moment}
When $q$ and $d$ tend to infinity with $d\geq 3p$, we have $M_2(q,d)=o(q^{1/3})$. 
\end{cor}
\begin{proof}
Since $X_i$'s are i.i.d. random variables, it is straightforward to compute that,
$$
{\mathbb{E}} \left( \left( \frac{1}{\sqrt{q+1}} \left( \sum_{i=1}^{q^2+q+1} X_i - (q+1) \right) \right)^2 \right)=1+o(1)
$$
as $q \to \infty$. Thus, in view of Theorem~\ref{moment}, it suffices to show that
\begin{equation}\label{d>Cp}
(d-1)^2 q^{\frac{5}{3}-\min\left(\left\lfloor \frac{d}{p}\right\rfloor+1, \frac{d}{3} \right)} + d q^{\frac{5}{3}-\left\lfloor\frac{d-1}{p}\right\rfloor-1}=o(1)
\end{equation}
when $q,d \to \infty$ and $d \geq 3p$.

When $p \leq 5000$ and $q \to \infty$, we have $q/p \to \infty$; thus, under the assumption that $d \to \infty$,
$$
(d-1)^2 q^{\frac{5}{3}-\min\left(\left\lfloor \frac{d}{p}\right\rfloor+1, \frac{d}{3} \right)} + d q^{\frac{5}{3}-\left\lfloor\frac{d-1}{p}\right\rfloor-1}=o(1).
$$

Next we assume $p>5000$ so that 
$\min\left(\left\lfloor \frac{d}{p}\right\rfloor+1, \frac{d}{3} \right)=\left \lfloor \frac{d}{p}\right\rfloor+1 \geq 4$ when $d \geq 3p$. Note that
$$
\log_q d \leq \frac{\log d}{\log p}= 1+ \frac{\log d/p}{\log p} \leq 1+\frac{\log d/p}{\log 5000}.
$$
When $x \geq 4$, it can be checked that,
$$
(\lfloor x \rfloor+1)-\frac{11}{3}-2\cdot \frac{\log x}{\log 5000}\geq x-\frac{11}{3}-2\cdot \frac{\log x}{\log 5000}>0.007.
$$
When $3 \leq x<4$, we have $\lfloor x \rfloor+1=4$ and again,
$$
(\lfloor x \rfloor+1)-\frac{11}{3}-2\cdot \frac{\log x}{\log 5000}=\frac{1}{3}-2\cdot \frac{\log x}{\log 5000}>0.007.
$$
Thus, it follows that
$$
\frac{5}{3}+2\log_q d-\left(\left\lfloor \frac{d}{p}\right\rfloor+1 \right)\leq \frac{11}{3}+2 \cdot \frac{\log d/p}{\log 5000}-\left(\left\lfloor \frac{d}{p}\right\rfloor+1 \right)<-0.007
$$
provided that $d \geq 3p$, which implies that 
$$
d^2q^{\frac{5}{3}-\left(\left\lfloor \frac{d}{p}\right\rfloor+1 \right)}=q^{2\log_q d+\frac{5}{3}-\left(\left\lfloor \frac{d}{p}\right\rfloor+1 \right)} \leq q^{-0.007}=o(1)
$$
when $q \to \infty$. 
Similarly, when $x \geq 3$, it is easy to verify that
$$
x-\frac{8}{3}-\frac{\log x}{\log 5000}>0.2.
$$
It follows that 
$$
dq^{\frac{5}{3}-d/p}=q^{\log_q d+\frac{5}{3}-d/p}<q^{-0.2}=o(1).
$$
Combining the ingredients above, we obtain
$$
(d-1)^2 q^{\frac{5}{3}-\min\left(\left\lfloor \frac{d}{p}\right\rfloor+1, \frac{d}{3} \right)} + d q^{\frac{5}{3}-\left\lfloor\frac{d-1}{p}\right\rfloor-1}
 \leq  d^2q^{\frac{5}{3}-\left(\left\lfloor \frac{d}{p}\right\rfloor+1 \right)}+dq^{\frac{5}{3}-d/p}=o(1)
$$
and the conclusion follows.
\end{proof}

\begin{rem}\label{rem: CLT}
By adapting the proof of the above corollary, we can show that if $q$ and $d$ tend to infinity such that $d/p \to \infty$, then as $C_F$ runs over smooth plane curves with degree $d$ over $\F_q$, the limiting distribution of 
$$
\frac{\# C_F(\F_q) - (q+1)}{\sqrt{q+1}}
$$
is the standard normal distribution (which is a stronger version of \cite{BDFL10}*{Corollary 1.4} whose hypothesis $d>q^{1+\epsilon}$ is more stringent). Under this stronger assumption, we can show that for each fixed positive integer $k$, the error term in Theorem~\ref{moment} is $o(1)$. The desired conclusion then follows from the triangular central limit theorem and the method of moments from probability theory; we refer to \cite{B95}*{Theorem 27.3} and \cite{BDFL10}*{Remark 1.5} for related discussions.
\end{rem}

\section{Linear independence, point-counting, and skew lines}\label{sect:skew-lines}

The purpose of this section is to apply a linear independence argument in the parameter space of plane curves to obtain results about point-counting and skew lines. The essential Proposition~\ref{prop:independence} will be used again in later sections. 

\subsection{Linear independence}
In algebraic geometry, we are often interested in understanding the space of degree $d$ hypersurfaces passing through a specified set of points with prescribed multiplicities. This is known as the \emph{interpolation problem}. The main difficulty is that when there are too many points relative to the degree $d$, vanishing at several points may no longer yield independent conditions. However, when the degree $d$ is large, the following proposition guarantees that different points impose linearly independent conditions on the parameter space. While the result is known to the experts (see \cite{P04}*{Lemma 2.1} where the argument is written in a cohomological language), our proof is self-contained and does not need advanced tools. 

\begin{prop}\label{prop:independence}
Fix a finite field $\mathbb{F}_q$, and consider any $k$ distinct  $\F_q$-points $P_1, P_2, \ldots, P_k$ in $\mathbb{P}^2$. If $d\geq \min\{k-1,2q-1\}$, then passing through $P_1, P_2, \ldots, P_k$ impose linearly independent conditions in the vector space of degree $d$ plane curves. 
\end{prop}

\begin{proof}
Let $V$ denote the vector space of all plane curves of degree $d$. Note that $\F_q$-points of $V$ exactly correspond to plane curves of degree $d$ defined over $\F_q$. It is known that $\dim(V)=N=\binom{d+2}{2}$. For each $1\leq i\leq k$, let 
$$
W_i = \{\text{plane curves of degree } d \text{ passing through } P_1, P_2, \ldots, P_i\}.
$$
We aim to show that $W_k$ has the expected dimension as an $\F_q$-vector space, namely $\dim(W_k) = \dim(V)-k$. It is clear that $W_i\supseteq W_{i+1}$ for each $1\leq i\leq k-1$. In particular, we have:
$$
\dim(W_1) \geq \dim(W_2) \geq \dim(W_3) \geq \cdots \geq \dim(W_{k}). 
$$
Moreover, passing through one additional point imposes \emph{at most} one linear condition, so we have either $\dim(W_{i+1})=\dim(W_i)-1$ or $\dim(W_{i+1})=\dim(W_{i})$ for each $i$. Since $\dim(W_1)=\dim(V)-1$, we must have $\dim(W_k) \geq \dim(V)-k$. 

In order to show that $\dim(W_k)=\dim(V)-k$, it suffices to show that the dimension drops by $1$ at each step; that is, we want to show that $\dim(W_{i+1})=\dim(W_i)-1$ holds for each $1\leq i\leq k-1$. To prove this, it suffices to find an element $f\in W_i$ such that $f\notin W_{i+1}$. 

We proceed to show that for $d\geq \min\{k-1,2q-1\}$, we can find a plane curve $C=\{f=0\}$ of degree $d$ defined over $\F_q$ such that $f\in W_{i}$ but $f\notin W_{i+1}$, that is, $f$ passes through $P_1, P_2, \ldots, P_{i}$ but does \emph{not} pass through $P_{i+1}$. 

First, assume that $d \geq k-1$. For each $j\in \{1, 2, \ldots, i\}$, let $L_j$ be an $\F_q$-line which passes through $P_j$ but not $P_{i+1}$. Then for any $d\geq k-1$, we can consider the reducible curve $C$ consisting of the lines $L_j$ (for $j=1, 2, \ldots, i$). By repeating some of the lines, we can ensure that $C$ has degree $d$. This is because $d\geq k-1$ and the number of lines $L_j$ (for $j=1, 2, \ldots, i$) is at most $i\leq k-1$. By construction, $C$ has degree $d$, $C$ passes through $P_1, P_2, \ldots, P_{i}$ but $C$ does not pass through $P_{i+1}$. 

Next, assume that $d \geq 2q-1$. In this case, we can instead use a construction due to Homma and Kim  (a special case of \cite{HK18}*{Corollary 2.4}).
We can find a polynomial $f$ with degree $d$ that vanishes on all points of $\bP^2(\F_q)$ except for $P_{i+1}$ (in particular, $C=\{f=0\}$ passes through $P_1, P_2, \ldots, P_{i}$ but $C$ does not pass through $P_{i+1}$) . In fact, we can explicitly construct such a polynomial $f$ by noticing that the polynomial
$$
x^{d-(2q-2)}(y^{q-1}-x^{q-1})(z^{q-1}-x^{q-1} )
$$
vanishes on $\bP^2(\F_q)\setminus \{[1:0:0]\}$. 

The proof of the strict inclusion $W_{i}\supsetneq W_{i+1}$ is complete, and the desired result follows. 
\end{proof}

\begin{rem}
When $k$ is close to $q^2+q+1$, the degree $2q-1$ in the statement of Proposition~\ref{prop:independence} is optimal. The paper by Homma and Kim~\cite{HK18}*{Section 3} contains a detailed discussion.
\end{rem}

When $d \geq q^2+q$, the following proposition \cite{BDFL10}*{Proposition 1.6} describes the explicit distribution of the number of rational points of degree $d$ plane curves over $\F_q$, which consequently implies that the limiting distribution of the number of rational points (up to normalization) is the standard normal distribution (compare to Theorem~\ref{moment} and Remark~\ref{rem: CLT}). The original proof of \cite{BDFL10}*{Proposition 1.6} relied on the more general interpolation result developed by Poonen~\cite{P04}*{Lemma 2.1}. Lercier et.~al. have observed that the construction due to Homma and Kim allows one to weaken the condition $d \geq q^2+q$ to $d \geq 2q-1$ in \cite{LRRSS19}*{Proposition 3.1}. Let us explain how Proposition~\ref{prop:independence} leads to a quick proof of this useful result. We present a complete proof below because we will refer to it in later discussions. 

\begin{prop}\label{prop:distribution}
Let $Y_1, \dots, Y_{q^2+q+1}$ be i.i.d. random variables taking the value $1$ with probability
$1/q$ and the value $0$ with probability $(q-1)/q$.
Then, for each integer $0 \leq t \leq q^2+q+1$ and $d \geq 2q-1$, we have 
$$\frac{\# \{ F \in S_d : \#C_F(\F_q)=t \}}{\# S_d} =
\Pr\left( Y_1 + \dots + Y_{q^2+q+1} = t \right).$$
\end{prop}

\begin{proof}
Enumerate all the points of $\bP^2(\F_q)$ as $P_1,P_2,\ldots, P_{q^2+q+1}$. Let $C$ be a random plane curve with degree $d$. We define the random variables $Z_1, Z_2,\ldots, Z_{q^2+q+1}$ as follows: if $C$ passes through $P_i$, assign $Z_i$ to be $1$ and otherwise $0$.  It suffices to show that for any $\epsilon_1, \epsilon_2,\ldots,\epsilon_{q^2+q+1} \in \{0,1\}$, we have
$$
\Pr(Z_i=\epsilon_i \text{ for all } i \leq q^2+q+1)= \Pr(Y_i=\epsilon_i \text{ for all } i \leq q^2+q+1)=\bigg (\frac{1}{q} \bigg)^{|I|} \bigg (\frac{q-1}{q} \bigg)^{q^2+q+1-|I|},
$$
where $I=\{i: \epsilon_i=1\}$. Since $d \geq 2q-1$, Proposition~\ref{prop:independence} implies that for any $J \subset \{1,2,\ldots, q^2+q+1\}$, 
$$
\Pr(Z_j=1 \text{ for all } j \in J)=\bigg (\frac{1}{q} \bigg)^{|J|}.
$$
It follows from the principle of inclusion and exclusion, and the binomial theorem that
\begin{align*}
&\Pr(Z_i=\epsilon_i \text{ for all } i \leq q^2+q+1)= \sum_{J \supset I} (-1)^{|J|-|I|} \Pr(Z_j=1 \text{ for all } j \in J)\\
&=\sum_{J \supset I} (-1)^{|J|-|I|} \bigg (\frac{1}{q} \bigg)^{|J|}
=\sum_{k=|I|}^{q^2+q+1} \binom{q^2+q+1-|I|}{k-|I|} (-1)^{k-|I|}\bigg (\frac{1}{q} \bigg)^{k}\\
&=\bigg (\frac{1}{q} \bigg)^{|I|} \sum_{k=0}^{q^2+q+1-|I|} \binom{q^2+q+1-|I|}{k} \bigg (-\frac{1}{q} \bigg)^{k}
=\bigg (\frac{1}{q} \bigg)^{|I|} \bigg (\frac{q-1}{q} \bigg)^{q^2+q+1-|I|}
\end{align*}
leading to the desired result. \end{proof}

\subsection{Expected number of skew lines}\label{subsect:poisson} In this subsection, we build machinery to calculate the number of skew lines to a random plane curve from a probabilistic perspective (see Corollary~\ref{cor:poisson}).

\begin{thm}\label{thm:poisson}
Suppose $f$ is a function with $f(q)/q\to 1$ as $q\to\infty$. Let $P_1, P_2, \ldots, P_{f(q)}$ be distinct points in $\bP^2(\F_q)$. For $d\geq f(q)-1$, consider the space $S_d$ of degree $d$ plane curves defined over $\F_q$. Let $X_q$ be the random variable representing the number of $P_i$'s contained in a random plane curve $C\in S_d$. As $q\to\infty$, the limiting distribution of $X_q$ is Poisson with mean $\lambda=1$. 
\end{thm}
\begin{proof}
We can apply Proposition~\ref{prop:independence}, similar to the proof of Proposition~\ref{prop:distribution}, to deduce that $X_q$ has a binomial distribution $\operatorname{Bin}(f(q), 1/q)$. The conclusion follows from Poisson limit theorem~\cite{B95}*{Theorem 23.2} since $$\lim_{q\to\infty} \mathbb{E}(X_q)= \lim_{q \to \infty} \frac{f(q)}{q}=1.$$
\end{proof}

As an application, we will discuss the following question: how many roots in $\F_q$ does a random polynomial in $\F_q[x]$ have? Leont'ev \cite{L06} proved the following result which shows that the Poisson distribution controls the number of distinct roots. The original proof is somewhat involved and uses probability-generating functions and some tools from complex analysis. We refer to \cite{LL21} for a different proof of this fact using tools from number theory and generating functions. The proof given below is much shorter.

\begin{cor}[Leont'ev]
As $q \to \infty$, the number of distinct $\F_q$-roots of a random polynomial $f \in \F_q[x]$ with a degree at most $q-1$, tends to the Poisson distribution with mean $\lambda=1$.
\end{cor}

\begin{proof}
Let $V_{q-1}$ be the subspace of $\F_q[x]$ consisting of polynomials with a degree at most $q-1$. There is a natural surjective linear map $\pi\colon S_{q-1} \to V_{q-1}$ given by $f(x,y,z)\mapsto f(x,1,0)$. Since any two fibers of $\pi$ have the same size, as we pick $f$ uniformly randomly from $S_{q-1}$, the resulting image polynomials $\pi(f)$ also result in the uniform distribution on $V_{q-1}$. Moreover, the number of distinct $\F_q$-roots of $\pi(f)$ is equal to the number of $\F_q$-points in the intersection of the curve $\{f=0\}$ with the affine line $\{z=0\}-\{[1:0:0]\}=\{[x:1:0]: x \in \F_q\}\cong \mathbb{A}^1$. The desired result follows from invoking Theorem~\ref{thm:poisson} with the $q$ points on this affine line.
\end{proof}

The next corollary is the first step in estimating the number of skew lines to a random plane curve. 

\begin{cor}\label{cor:poisson}
Let $L$ be a fixed $\F_q$-line. For a random plane curve $C$ of degree $d\geq q$, the random variable $\# (C \cap L)(\F_q)$ tends to the Poisson distribution with mean $\lambda=1$. In particular, the expected number of skew lines to a random plane curve $C$ is $q^2/e+o(q^2)$.
\end{cor}

\begin{proof}
The first assertion follows immediately from Theorem~\ref{thm:poisson}. In particular, for each line $L$, the probability that $(C \cap L)(\F_q)=\emptyset$ is $1/e+o(1)$. Since there are $q^2+q+1$ lines, the expected number of skew lines to $C$ is $(q^2+q+1) \cdot (1/e+o(1))=q^2/e+o(q^2)$ by the linearity of expectation. 
\end{proof}

In Theorem~\ref{thm:expected-skew-lines}, we will provide a refined estimate on the expected number of skew lines.

\begin{rem}\label{rmk:entin}
Corollary~\ref{cor:poisson} holds when $d$ is large compared to $q$. On the other hand, Entin~\cite{E21} obtained the same result on the number of skew lines when $q$ is large compared to $d$. Let $C\subset\bP^2$ be a geometrically irreducible projective plane curve of degree $d$. Consider the space $Y=(\bP^2)^{\ast}$ which parametrizes all lines in $\mathbb{P}^2$. We also have the variety $X$ given by:
$$
X = \{ (P, L) \ | \ P\in C\cap L \} \subset \bP^2\times (\bP^2)^{\ast}.
$$
Consider the map $f\colon X\to Y$ given by projection to the second coordinate. Note that $f$ is a finite map with $\deg(f)=d$. Suppose that the monodromy group $\operatorname{Mon}(X/Y)$ is equal to the symmetric group $\mathfrak{S}_d$. Having an $\F_q$-point in the fiber $f^{-1}(L)$ for a given $\F_q$-line $L$ is equivalent to a corresponding Frobenius class of the map $f$ at the point $L$ having a fixed point. Applying the Chebotarev density theorem for varieties over finite fields \cite{E21}*{Theorem 3} for each conjugacy class in $\mathfrak{S}_d$ with no fixed points and summing up the contributions, we obtain:
$$
 \frac{\#\{\text{skew }\F_q\text{-lines to } C\}}{q^2} \approx \frac{\# \text{permutations without fixed points}}{\# \mathfrak{S}_d} = \frac{!d}{d!} = \sum_{i=0}^{d} \frac{(-1)^{i}}{i!}
$$
for sufficiently large $q$. As $d$ gets larger and larger (but fixed), the proportion above will approach $1/e$. When $d$ is fixed, and $q\to\infty$, almost all plane curves are smooth (due to the Lang-Weil bound applied to the discriminant hypersurface) and, in particular, geometrically irreducible. Therefore, the assumption made above on the curve does not affect the density if we consider all plane curves of degree $d$. This limiting proportion of $1/e$ is in agreement with our result above. However, the ranges for $q$ and $d$ are different since we allow $d$ to be large with respect to $q$.

More generally, for a fixed $k\in\mathbb{N}$, the expected density of $\mathbb{F}_q$-lines meeting $C$ at exactly $k$ distinct $\mathbb{F}_q$-points is $1/(e\cdot k!)$. We can prove this in two different scenarios using similar arguments: (1) When $q$ is large with respect to $d$, we apply Entin's result and choose the conjugacy class with $k$ fixed points, and (2) When $d$ is large compared to $q$, we use the following well-known fact: the number of fixed points of a random permutation on $d$ letters converges to the Poisson distribution with mean $1$ as $d$ tends to infinity \cite{K07}*{Example 1.19}. 
\end{rem}

\section{Proportion of blocking plane curves}\label{sect:proportion-general-case}

In this section, we will analyze the asymptotic proportion of plane curves of degree $d$ over $\F_q$ that do not form a blocking set. One of the goals of this section is to explain that even a curve of large degree $d$ over a fixed finite field $\F_q$ is not blocking with a very high probability.

Let $\F_q$ be a finite field, and $S=\F_q[x,y,z]$. As we saw in Section~\ref{sect:prelim}, $S_d\subset S$ denotes the vector subspace of polynomials of degree $d\geq 1$ (including the $0$ polynomial). For any subset $A\subset S$, let $\mu_{d}(A) = \frac{\# (A\cap S_d)}{\# S_d}$. We can define the \emph{upper natural density} and \emph{lower natural density} respectively as follows,
$$
\overline{\mu}(A) = \limsup_{d\to\infty} \mu_d(A) = \limsup_{d\to\infty} \frac{\# (A\cap S_d)}{\# S_d}, \ \ \ \ \ \ \ \ \underline{\mu}(A) = \liminf_{d\to\infty} \mu_d(A) = \liminf_{d\to\infty} \frac{\# (A\cap S_d)}{\# S_d}.
$$
If the lower and upper densities agree, then we define the \emph{natural density} of $A$ by $\mu(A)=\overline{\mu}(A)=\underline{\mu}(A)$ and in this case, 
$$
 \mu(A) = \lim_{d\to\infty} \frac{\# (A\cap S_d)}{\# S_d}.
$$
Recall that $C_F$ denotes the plane curve defined by the equation $F=0$. Let $\mathcal{B}\subset S=\F_q[x,y,z]$ denote the set of homogeneous polynomials $F$ such that the curve $C_F$ is blocking. The primary goal of the present paper is to understand the natural density of $\mathcal{B}$. The complement of $\mathcal{B}$ is easier to work with, justifying the following.

\begin{defn}
Let $\F_q$ be a fixed finite field. Define $\mathcal{E}$ as follows:
$$
\mathcal{E} = \bigcup_{d\geq 1} \{ F\in S_d \ | \ C_F \text{ is not blocking} \}.
$$
\end{defn}

To gain insight into the natural density $\mu(\mathcal{E})$, we first stratify the space $\mathcal{E}$ by different $\F_q$-lines and understand the corresponding natural density $\mu(\mathcal{E}_L)$.

\begin{prop}\label{prop:one-line}
Let $L$ be a fixed $\F_q$-line in $\bP^2$. Let 
$$
\mathcal{E}_L = \bigcup_{d\geq 1} \{F\in S_d \ | \ C_F\cap L \text{ has no } \F_q\text{-points}\}.
$$
Then $\displaystyle \mu_d(\mathcal{E}_L)=\left(1-\frac{1}{q}\right)^{q+1}$ for $d\geq q$, which is approximately $\dfrac{1}{e}$ for $q$ large. 
\end{prop} 

\begin{proof}
Let 
$$
\mathcal{A}_d = \{ F\in S_d \ | \ C_F\cap L \text{ has at least one } \F_q\text{-point}\}. 
$$
For each $\F_q$-point $P\in L$, define:
$$
\mathcal{A}_d^{P} = \{ F\in S_d \ | \ P\in C_F\cap L\}. 
$$
It is clear that
$$
\mathcal{A}_d = \bigcup_{P\in L(\F_q)} \mathcal{A}_{d}^{P}.
$$
Let $N=\binom{d+2}{2}$. For any $1\leq k\leq q+1$ points $P_1, P_2, \ldots, P_k\in L(\F_q)$, by Proposition~\ref{prop:independence},
$$
\# \bigcap_{i=1}^{k} \mathcal{A}_d^{P_i} = q^{N-k}
$$
for $d\geq k-1$. Using the principle of inclusion and exclusion, we obtain:
\begin{align*}
\# \mathcal{A}_d &= \# \bigcup_{P\in L(\F_q)} \mathcal{A}_{d}^{P} = \sum_{k=1}^{q+1} \binom{q+1}{k} (-1)^{k+1}\cdot \left( q^{N-k} \right) \\
&=  q^{N}\sum_{k=1}^{q+1}\binom{q+1}{k}(-1)^{k+1} (q^{-1})^k = q^N +  q^{N}\sum_{k=0}^{q+1} \binom{q+1}{k}(-1)^{k+1} (q^{-1})^k \\
&= q^N -  q^{N}\sum_{k=0}^{q+1} \binom{q+1}{k} (-q^{-1})^k = q^N - q^N(1-q^{-1})^{q+1}
\end{align*}
for $d\geq q$. Thus,
$$
\mu_d(\mathcal{A}_d) = \frac{q^N - q^N(1-q^{-1})^{q+1}}{q^N} = 1-(1-q^{-1})^{q+1}.
$$
Therefore,
$$
\mu_d(\mathcal{E}_L) =1-\mu_d(\mathcal{A}_d) = 1 - (1-(1-q^{-1})^{q+1}) = (1-q^{-1})^{q+1} \approx \frac{1}{e}
$$
as desired. \end{proof}

Proposition~\ref{prop:one-line} already allows us to deduce one of our main results, namely Theorem~\ref{thm:expected-skew-lines}, on the expected number of skew lines.

\begin{proof}[Proof of Theorem~\ref{thm:expected-skew-lines}] As a preliminary step,  
we observe the Taylor expansion,
\begin{equation}\label{eq:taylor}
\bigg(1-\frac{1}{x}\bigg)^x-\frac{1}{e}=-\frac{1}{2ex}-\frac{5}{24ex^2}+O\bigg(\frac{1}{x^3}\bigg).
\end{equation}
By Proposition~\ref{prop:one-line}, the probability that a fixed line $L$ is skew to a random plane curve $C$ of degree $d\geq q$ is exactly $\left(1-\frac{1}{q}\right)^{q+1}$. Using \eqref{eq:taylor}, we obtain
$$
\bigg(1-\frac{1}{q}\bigg)^{q+1} =\bigg(\frac{1}{e}-\frac{1}{2eq}-\frac{5}{24eq^2}+O\bigg(\frac{1}{q^3}\bigg)\bigg)\bigg(1-\frac{1}{q}\bigg)=\frac{1}{e}-\frac{3}{2eq}+\frac{7}{24eq^2}+O\bigg(\frac{1}{q^3}\bigg).
$$
Thus, by the linearity of expectation, the expected number of skew lines to $C$ is
$$
\frac{q^2+q+1}{e}-\frac{3(q^2+q+1)}{2eq}+\frac{7(q^2+q+1)}{24eq^2}+O\bigg(\frac{1}{q}\bigg)=\frac{q^2}{e}-\frac{q}{2e}-\frac{5}{24e}+O\bigg(\frac{1}{q}\bigg)
$$
as desired. \end{proof}

The following result gives an explicit formula for the natural density $\mu_d(\mathcal{E})$ in terms of the statistics on the number of $\F_q$-points contained in a union of $\F_q$-lines.

\begin{prop}\label{proportion-of-curves}
Let $\F_q$ be a finite field. Let $(\bP^2)^{\ast}(\F_q)$ denote the set of all $\F_q$-lines in the plane. For $d\geq 2q-1$, we have
$$
\mu_d(\mathcal{E}) = \sum_{k=1}^{q^2+q+1} \sum_{\substack{T\subset (\bP^2)^{\ast}(\F_q) \\ |T|=k}}  (-1)^{k+1}\cdot \left(1-\frac{1}{q}\right)^{\# (\bigcup_{L\in T} L)(\F_q) }.
$$
In particular, the formula above also computes the constant $\nb(q)=\displaystyle\lim_{d\to\infty} \mu_d(\mathcal{E})$. 
\end{prop}

\begin{proof}
Note that $F\in \mathcal{E}$ if and only if $C_F\cap L$ has no $\F_q$-points for some $\F_q$-line $L$. Using the notation used in Proposition~\ref{prop:one-line}, we have
$$
\mathcal{E} = \bigcup_{L\in(\bP^2)^{\ast}(\F_q)}\mathcal{E}_L.
$$
Using the principle of inclusion and exclusion, we get:
$$
|\mathcal{E}\cap S_d| = \sum_{k=1}^{q^2+q+1} \sum_{\substack{T\subset (\bP^2)^{\ast}(\F_q) \\ |T|=k}}  (-1)^{k+1} \cdot \left|  \bigcap_{L\in T} \mathcal{E}_L \cap S_d\right|.
$$
It remains to calculate the density of $\bigcap_{L\in T} \mathcal{E}_L$. Using the same idea as in Proposition~\ref{prop:distribution} and Proposition~\ref{prop:one-line}, we obtain:
$$
\mu_d\left(\bigcap_{L\in T} \mathcal{E}_L\right) = \left(1-\frac{1}{q}\right)^{\# (\bigcup_{L\in T} L)(\F_q) }
$$
as desired. \end{proof}

Even though Proposition~\ref{proportion-of-curves} gives us the theoretical answer for the proportion of non-blocking curves, it is not feasible in practice. The inequalities in our main Theorem~\ref{thm:lower-upper-bounds} are more applicable. We present the proof of these inequalities by estimating the ``main term" contribution to the sum in Proposition~\ref{proportion-of-curves}. 

\begin{proof}[Proof of Theorem~\ref{thm:lower-upper-bounds}]
We start with the proof of the upper bound on $\nb(q)$. Our goal is to prove the following:
\begin{align}\label{ineq:1-nb(q)}
1-\nb(q) \geq \frac{1}{q} \left(1-\left(1-\left(\frac{1}{q}\right)^q\right)^{q+1}\right) +\frac{1}{q} \sum_{j=1}^q  \left(1-\frac{1}{q}\right)^j \left(1-\left(1-\left(\frac{1}{q}\right)^q\right)^{q}\right).
\end{align}
Note that $1-\nb(q)$ is the density of blocking curves. On the other hand, we claim that the right-hand side of inequality \eqref{ineq:1-nb(q)} represents the density of a particular family of (trivially) blocking curves. To see this, consider $q+1$ collinear $\F_q$-points $P_0, \ldots, P_{q}$ on a fixed $\F_q$-line $L_0$. Let us say that a plane curve $C$ is \emph{trivially blocking with respect to} a point $P$ if $C$ contains all the $\F_q$-points of \emph{some} $\F_q$-line $L$ passing through $P$. For each $0\leq j\leq q$, consider
$$
\mathcal{B}_{j} = \{F \in \mathcal{B} \ | \ C_F \text{ is trivially blocking with respect to } P_j, \text{ and } P_i\notin C_F \text{ for } i<j\}.
$$
Then by the same idea as in Proposition~\ref{prop:distribution},
\begin{align*}
\mu(\mathcal{B}_0) &= \frac{1}{q} \left(1-\left(1-\left(\frac{1}{q}\right)^q\right)^{q+1}\right), \\ 
\mu(\mathcal{B}_j) &= \frac{1}{q}  \left(1-\frac{1}{q}\right)^j \left(1-\left(1-\left(\frac{1}{q}\right)^q\right)^{q}\right)  \quad \quad \text{for } j\geq 1.
\end{align*}
Indeed, $\mu(\mathcal{B}_0)$ counts the density of curves that are trivially blocking with respect to $P_0$. First, we restrict ourselves to the event where the curve passes through $P_0$. Given a line $L\ni P_0$, note that $1-\left(\frac{1}{q}\right)^q$ is the density of curves $C$ such that $L(\F_q)\not\subset C(\F_q)$. By independence, $\left(1-\left(\frac{1}{q}\right)^{q}\right)^{q+1}$ is the density of curves for which all the $q+1$ lines passing through $P$ satisfy $L(\F_q)\not\subset C(\F_q)$. We obtain the formula for $\mu(\mathcal{B}_0)$ by looking at the complement. The analogous argument works for $\mu(\mathcal{B}_j)$ for $j\geq 1$ with the exception that one of the $\F_q$-lines passing through $P_j$ is $L_0$ which already satisfies $L_0(\F_q)\not\subset C(\F_q)$ so we can ignore it while counting the $\F_q$-lines passing through $P_j$. As the subsets $\mathcal{B}_j$ for $j=0, 1, \ldots, q$ are disjoint, the inequality \eqref{ineq:1-nb(q)} readily follows.

The proof of the lower bound on $\nb(q)$ proceeds similarly. To give a quick overview, we want to establish 
\begin{align}\label{ineq:nb(q):geq}
\nb(q) \geq \left(1-\frac{1}{q}\right)\left(1 - \left(1- \left(1-\frac{1}{q}\right)^{q}\right)^{q+1}\right) +\left(1-\frac{1}{q}\right) \sum_{j=1}^q \left(\frac{1}{q}\right)^j\left(1 - \left(1- \left(1-\frac{1}{q}\right)^{q}\right)^q\right)
\end{align}
Let us say that a plane curve $C$ is \emph{skew with respect to} a point $P$ if $C$ contains none of the $\F_q$-points of \emph{some} $\F_q$-line $L$ passing through $P$. For $0 \leq j \leq q$, consider the disjoint sets:
$$
\mathcal{D}_{j} = \{F \in \mathcal{E} \ | \ C_F \text{ is skew with respect to } P_j, \text{ and } P_i\in C_F \text{ for } i<j\}.
$$
The first term in the right-hand side of inequality \eqref{ineq:nb(q):geq} comes from $\mu(\mathcal{D}_0)$, and the $j$-th summand in the sum is due to $\mu(\mathcal{D}_j)$ for $j=1, \ldots, q$. The difference between the cases $j=0$ and $j\geq 1$ is reflected by the fact that $L_0$ cannot be skew to $C_F$ for  $F\in \mathcal{D}_j$ once $j\geq 1$. This completes the proof of the lower bound \eqref{ineq:nb(q):geq}. \end{proof}

The discussion above leads to an instant proof of our first main theorem.

\begin{proof}[Proof of Theorem~\ref{thm:most-curves-are-not-blocking}]
The result immediately follows from the lower bound in Theorem~\ref{thm:lower-upper-bounds} and the last assertion in Proposition~\ref{proportion-of-curves}.
\end{proof}

While we know that $\lim_{q\to\infty} \nb(q)=1$, it is less clear how $\nb(q)$ and $\nb(q')$ compare for different values of $q$ and $q'$. Based on our various computations (see also Section~\ref{sect:proportion-small-fields} in which we will calculate the exact values of $\nb(2), \nb(3)$, and $\nb(4)$), we advance the following question.
\begin{question}\label{quest:nb-increasing}
Is $\nb(q)$ an increasing function of $q$?
\end{question}

We observe that Theorem~\ref{thm:lower-upper-bounds} yields an increasing sequence of prime powers $q_1 < q_2 < q_3 < \cdots$ such that $\nb(q_1) < \nb(q_2) < \nb(q_3) < \cdots$, providing some evidence for a positive answer to Question~\ref{quest:nb-increasing}.

\section{Proportion of smooth blocking curves}\label{sect:proportion-smooth-case}

Let $q$ be a fixed prime power. When $d \to \infty$, Poonen's result \cite{P04}*{Theorem 1.1} implies that the density of smooth curves over $\F_q$ is $\zeta_{\PP^2}(3)^{-1}=(1-1/q)(1-1/q^2)(1-1/q^3)$. The following lemma \cite{BDFL10}*{Equation 8} provides an effective version with an explicit error bound.

\begin{lem}\label{lem: smooth}
We have
\begin{align*}
\frac{\#S_{d}^{\mathrm{ns}}}{\# S_{d}}=& \bigg(1-\frac{1}{q}\bigg)\bigg(1-\frac{1}{q^2}\bigg)\bigg(1-\frac{1}{q^3}\bigg) \left(1+O\left(\frac{d^{-1 / 3}}{1-q^{-1}-2 d^{-1 / 3}}\right)\right) \\
&+O\left(\frac{d^{-1 / 3}}{1-q^{-1}}+(d-1)^{2} q^{-\min \left(\left\lfloor\frac{d}{p}\right\rfloor+1, \frac{d}{3}\right)}+d q^{-\left\lfloor\frac{d-1}{p}\right\rfloor-1}\right).
\end{align*}
\end{lem}

\begin{cor}\label{cor: smooth}
As $q,d\to \infty$ and $d\geq \min \{2q,3p\}$, we have
$$
\frac{\#S_{d}^{\mathrm{ns}}}{\# S_{d}} \geq \frac{21}{64}+o(1).
$$
\end{cor}
\begin{proof}
When $q,d\to \infty$, it is clear that
$$
\frac{d^{-1 / 3}}{1-q^{-1}} \to 0, \quad \frac{d^{-1 / 3}}{1-q^{-1}-2 d^{-1 / 3}} \to 0.
$$
Moreover, under the extra condition $d\geq 3p$, we have shown in equation~\eqref{d>Cp} that 
$$
(d-1)^{2} q^{-\min \left(\left\lfloor\frac{d}{p}\right\rfloor+1, \frac{d}{3}\right)}+d q^{-\left\lfloor\frac{d-1}{p}\right\rfloor-1} \to 0.
$$
If instead we have $d \geq 2q$, one can prove the same estimate by mimicking the proof of Corollary~\ref{second_moment} by separating the analysis into two cases $p \leq 11$ and $p>11$.

Thus, the statement follows immediately from Lemma~\ref{lem: smooth} and the fact that
$$
\bigg(1-\frac{1}{q}\bigg)\bigg(1-\frac{1}{q^2}\bigg)\bigg(1-\frac{1}{q^3}\bigg) \geq \bigg(1-\frac{1}{2}\bigg)\bigg(1-\frac{1}{2^2}\bigg)\bigg(1-\frac{1}{2^3}\bigg)=\frac{21}{64}
$$
holds for all $q$.
\end{proof}

Now we are ready to prove Theorem~\ref{thm:most-smooth-curves-are-not-blocking}.

\begin{proof}[Proof of Theorem~\ref{thm:most-smooth-curves-are-not-blocking}]

When $d \geq 2q$, the result follows immediately from combining Theorem~\ref{thm:most-curves-are-not-blocking} and Corollary~\ref{cor: smooth}. We assume that $d \geq 3p$ for the remainder of the proof.

Let $C$ be a random smooth plane curve of degree $d$ defined over $\F_q$, uniformly chosen from $S_d^{\mathrm{ns}}$. In view of Theorem~\ref{thm: lb}, the following three conditions would guarantee that $C(\F_q)$ is not a blocking set:
\begin{itemize}
    \item $C(\F_q)$ does not contain an $\F_q$-line.
    \item If $q$ is a square and $q>16$, then $C(\F_q)$ does not contain a Baer subplane.
    \item $\#C(\F_q)< q+1+q^{2/3}/2^{1/3}$.
\end{itemize}
If $d< q+1$, then the first condition is automatically true since it follows from B\'ezout's theorem that $\# (C\cap L) < q+1$ for any line $L$ defined over $\F_q$. If $d \geq q+1$, then for any given $L$, Proposition~\ref{prop:independence} implies that the proportion of degree $d$ plane curves defined over $\F_q$ which contain the line $L$ is $(\frac{1}{q})^{q+1}$. Since there are in total $q^2+q+1$ lines, it follows from the union bound and Corollary~\ref{cor: smooth} that
$$
\Pr(\text{$C(\F_q)$ contains some $\F_q$-line}) \leq \bigg(\frac{64}{21}+o(1)\bigg) (q^2+q+1)\bigg(\frac{1}{q}\bigg)^{q+1}=o(1).
$$

As for the second condition, it is known that each Baer subplane in $\bP^2(\F_q)$ is isomorphic to $\bP^2(\F_{\sqrt{q}}) \subset \bP^2(\F_q)$ and thus the number of Baer subplanes is bounded by a polynomial in $q$. More precisely, the number of Baer subplanes in $\bP^2(\F_q)$ is $q^{3/2}(q^{3/2}+1)(q+1)$ \cite{M00}*{Corollary 1.3}. Note that each Baer subplane has $q+\sqrt{q}+1>5$ points when $q>16$, so if $d \geq 5$, Proposition~\ref{prop:independence} implies that the proportion of degree $d$ plane curves defined over $\F_q$ which contain a given Baer subplane $B$ is at most $(\frac{1}{q})^{5}$ since the curve must pass through any given 5 points in $B$ in order to contain $B$. It follows from the union bound and Corollary~\ref{cor: smooth} that
$$
\Pr(\text{$C(\F_q)$ contains some Baer subplane}) \leq \bigg(\frac{64}{21}+o(1)\bigg) q^{3/2}(q^{3/2}+1)(q+1)\bigg(\frac{1}{q}\bigg)^{5}=o(1).
$$

Regarding the third condition, we apply Corollary~\ref{second_moment} and Chebyshev's inequality:
\begin{align*}
\Pr\left( \# C(\F_q)\geq q+1+q^{2/3}/2^{1/3} \right)
&=\Pr\left( \frac{\# C(\F_q) - (q+1)}{\sqrt{q+1}}\geq \frac{q^{2/3}}{2^{1/3}\sqrt{q+1}} \right)\\
&\leq \frac{M_2(q,d)\cdot 2^{2/3}(q+1)}{q^{4/3}} \cdot (1+o(1))=o(1).
\end{align*}
While applying Chebyshev's inequality, we used the fact that the mean of $(\# C(\F_q) - (q+1))/\sqrt{q+1}$ is $M_1(q,d)=o(1)$ from Theorem~\ref{moment}.

We conclude that
\begin{align*}
&\Pr(\text{$C$ is not blocking})
\geq 1-\Pr(\text{$C(\F_q)$ contains some $\F_q$-line})\\
&- \Pr(\text{$C(\F_q)$ contains some Baer subplane})-
\Pr\left( \# C(\F_q)\geq q+1+q^{2/3}/2^{1/3} \right)=1-o(1),
\end{align*}
as required.
\end{proof}

\begin{rem}
Using an argument similar to the above proof, we can instead use Proposition~\ref{prop:distribution} to deduce Theorem~\ref{thm:most-curves-are-not-blocking}. In fact, we can get an explicit lower bound on $\nb(q)$. However, even if we use the Chernoff bound (for example, Bernstein inequality) instead of Chebyshev's inequality to get a sharper estimate on the tail distribution $\Pr\left( \# C(\F_q)\geq q+1+q^{2/3}/2^{1/3} \right)$, the lower bound on $\nb(q)$ obtained in this manner is still much worse compared to the combinatorial lower bound from Section~\ref{sect:proportion-general-case}.
\end{rem}

\begin{cor}
Let $r$ be an integer such that $r \geq 7$. Let $q=p^r$. Then for any sequence $\{d_p\}$, as $p \to \infty$, almost all degree $d_p$ smooth curves over $\F_q$ are not blocking.
\end{cor}

\begin{proof}
We can decompose the sequence into two parts: the parts with $d_p^6<q$, that is, $d_p<p^{r/6}$ (such low degree curves are not blocking by \cite{AGY23}*{Theorem 1.2}); the parts with $d_p>p^{r/6}$. Thus, it suffices to consider any subsequence $\{p_n\}_{n=1}^{\infty}$ of the primes, such that $d_{p_n}>p_n^{r/6}$. However, note that $d_{p_n}/p_n \to \infty$ since $r \geq 7$. Consequently,  the density is also $1$ as $n \to \infty$ by Theorem~\ref{thm:most-smooth-curves-are-not-blocking}.
\end{proof}

\section{Independence between being smooth and being blocking} \label{sect:independence-smooth-blocking}

In this section, we investigate the probabilistic independence between the properties of being smooth and blocking. Given that we obtained a stronger result for the density of smooth blocking curves in Section~\ref{sect:proportion-smooth-case}, it is natural to ask whether the probability of a random plane curve being blocking is affected when conditioning on the event that the curve is smooth. We will show that for a fixed $\F_q$ with $q$ large, it is more likely for a smooth curve to be non-blocking compared to a general curve. However, we will show that these two probabilities are asymptotically equal as $q \to \infty$. Recall that 
\begin{equation*}
\nb(q) = \lim_{d\to\infty} \frac{\# \{F\in S_d : C_F \text{ is not blocking}\}}{\# S_d},
\end{equation*}
and
\begin{equation*}
\nb^{\mathrm{ns}}(q) = \lim_{d\to\infty} \frac{\# \{F\in S^{\mathrm{ns}}_d : C_F \text{ is not blocking}\}}{\# S^{\mathrm{ns}}_d}.
\end{equation*}
We introduce further notations to simplify the upcoming proof. For each set $B\subseteq\bP^2(\F_q)$, we define
\begin{equation}
\nu(q, B) = \lim_{d\to\infty}  \frac{\#\{F \in S_{d}: C_F(\F_q)=B\}}{\# S_{d}}.
\end{equation}
Similarly, we define
\begin{equation}
 \nu^{\operatorname{ns}}(q, B) = \lim_{d \to \infty} \frac{\#\{F\in S_d^{\mathrm{ns}} :  C_F(\F_q)=B\}}{\#S_d^{\mathrm{ns}}}.
\end{equation}

We prove a more refined version of Theorem~\ref{thm:smooth-blocking-independence} from the introduction. 

\begin{thm}\label{thm:independence-refined}
For sufficiently large $q$, we have $\nb(q)<\nb^{\mathrm{ns}}(q)$, and therefore, the events of being smooth and being blocking are not independent for curves over $\F_q$. However,
$$
\lim_{q\to\infty} \frac{1-\nb^{\mathrm{ns}}(q)}{1-\nb(q)}=1.
$$
Thus, the events of being smooth and being blocking are asymptotically independent.
\end{thm}

\begin{proof} Since $q$ is fixed, there are only finitely many configurations of blocking sets $B$ in $\bP^2(\F_q)$. Thus, we can interchange the limit and summation to get
$$
1-\nb(q)=\sum_{B \subseteq \bP^2(\F_q) \text{ blocking set}} \nu(q,B).
$$
Similarly, we have
$$
1-\nb^{\operatorname{ns}}(q)=\sum_{B \subseteq \bP^2(\F_q) \text{ blocking set}} \nu^{\operatorname{ns}}(q,B).
$$
Let $A_{i}$ and $A^{\mathrm{ns}}_{i}$ be the contribution to $1-\nb(q)$ and $1-\nb^{\mathrm{ns}}(q)$ from blocking sets with size $i$, respectively. 

Let $B$ be a blocking set with $|B|=t$; then we know $t \geq q+1$. By \cite{BDFL10}*{p. 2537}, we have
\begin{equation}\label{eq:mu-ns-q-B}
\nu^{\operatorname{ns}}(q, B) = \left( \frac{q+1}{q^2+q+1} \right)^{t} \left( \frac{q^2}{q^2+q+1} \right)^{q^2+q+1-t}.
\end{equation}
On the other hand, in the proof of Proposition~\ref{prop:distribution}, we have shown that
\begin{equation}\label{eq:mu-q-B}
\nu(q, B)=\bigg(\frac{1}{q} \bigg)^{t} \bigg (\frac{q-1}{q} \bigg)^{q^2+q+1-t}.
\end{equation}
It follows from equations~\eqref{eq:mu-ns-q-B} and \eqref{eq:mu-q-B} that
\begin{equation}\label{eq:ratio-nu-ns-to-nu}
\frac{\nu^{\mathrm{ns}}(q, B)}{\nu(q, B)} =\bigg(1-\frac{1}{q^2} \bigg)^{t} \bigg (1+\frac{1}{q^3-1} \bigg)^{q^2+q+1}.
\end{equation}
Note that the ratio is a decreasing function in $t$. When $t=o(q^2)$, the above ratio is close to $1$. However, if $t \approx Cq^2$, the above ratio is close to $1/e^C<1$. Summing over the contribution from all blocking sets $B$, this already implies that
\begin{equation}\label{ineq:1/e-and-1}
\frac{1}{e} \leq \liminf_{q \to \infty} \frac{1-\nb^{\mathrm{ns}}(q)}{1-\nb(q)}\leq \limsup_{q \to \infty} \frac{1-\nb^{\mathrm{ns}}(q)}{1-\nb(q)}\leq 1.
\end{equation}

We first show $\nb^{\mathrm{ns}}(q)>\nb(q)$, or equivalently, $1-\nb(q)>1-\nb^{\mathrm{ns}}(q)$, when $q$ is sufficiently large. Using the Taylor expansion, it is easy to see that, when $q$ is sufficiently large, the ratio in equation~\eqref{eq:ratio-nu-ns-to-nu} is less than $1$ whenever $t=|B| \geq q+2$. Thus, $A_i>A_i^{\mathrm{ns}}$ for $i \geq q+2$. It suffices to show that,
\begin{equation}\label{eq:q+1-vs-q+4}
A_{q+4}+A_{q+1}> A_{q+4}>A^{\mathrm{ns}}_{q+4}+A^{\mathrm{ns}}_{q+1}.
\end{equation}

By Theorem~\ref{thm: lb}, a blocking set with size $q+1$ or $q+4$ must be a trivial blocking set. Thus, the number of blocking sets with size $q+1$ is $q^2+q+1$ (the number of $\F_q$-lines), and the number of blocking sets with size $q+4$ is $(q^2+q+1) \cdot \binom{q^2}{3}$ (pick an $\F_q$-line and then pick three points outside the line). This observation, together with equations~\eqref{eq:mu-ns-q-B} and \eqref{eq:mu-q-B}, allows one to deduce that the contribution from blocking sets with size $q+1$ is negligible compared to the contribution from blocking sets with size $q+4$ when $q$ is large. More precisely, when $q$ is sufficiently large, from equation~\eqref{eq:ratio-nu-ns-to-nu}, we  obtain that $A_{q+4} > A^{\mathrm{ns}}_{q+4}(1+1/q^2)$ and $A^{\mathrm{ns}}_{q+1}=o(A^{\mathrm{ns}}_{q+4}/q^2)$. Thus, $A_{q+4}+A_{q+1}> A_{q+4}>A^{\mathrm{ns}}_{q+4}+A^{\mathrm{ns}}_{q+1}$ when $q$ is large, confirming equation~\eqref{eq:q+1-vs-q+4}. Now, $1-\nb(q) > 1-\nb^{\mathrm{ns}}(q)$ follows from equation~\eqref{eq:q+1-vs-q+4} and the previous observation $A_i>A_i^{\mathrm{ns}}$ for $i \geq q+2$.

We proceed to prove the second assertion regarding asymptotic independence. We know that $1-\nb(q) \geq (\frac{1}{q})^{q+1}$ by Theorem~\ref{thm:lower-upper-bounds}. We claim that the contribution to $1-\nb(q)$ coming from blocking sets with size greater than $q^{2}/\log q$ is negligible:
\begin{equation}\label{eq:large-blocking-sets}
\sum_{i \geqslant q^2/\log q} A_i
\leq 
\sum_{\substack{B \subseteq \bP^2(\F_q),\\ |B| \geqslant q^{2}/\log q}} \bigg(\frac{1}{q} \bigg)^{|B|}\leq 2^{q^2+q+1} \bigg(\frac{1}{q} \bigg)^{q^{2}/\log q}=\bigg (\frac{2+o(1)}{e}\bigg)^{q^2}=o\bigg(\frac{1}{q^{q+1}}\bigg).   
\end{equation}

Combining equations \eqref{eq:ratio-nu-ns-to-nu} and \eqref{eq:large-blocking-sets}, we obtain the following lower bound:
\begin{align*}
1-\nb^{\mathrm{ns}}(q)
&\geq \sum_{i \leqslant q^2/\log q}
A_i^{\mathrm{ns}} \geq \sum_{i\leqslant q^2/\log q} \bigg(1-\frac{1}{q^2} \bigg)^{q^2/\log q} \bigg (1+\frac{1}{q^3-1} \bigg)^{q^2+q+1} A_i  \\
&\geq \bigg(1-\frac{1}{q^2} \bigg)^{q^2/\log q} \bigg (1+\frac{1}{q^3-1} \bigg)^{q^2+q+1} \bigg(1-\nb(q)-\sum_{i \geqslant q^2/\log q}
A_i \bigg)\\
&=\bigg(1-O\bigg(\frac{1}{\log q}\bigg)\bigg)\bigg(1+O\bigg(\frac{1}{q}\bigg)\bigg) \big(1-\nb(q)\big)\big(1-o(1)\big).
\end{align*}
Consequently,
$$
\liminf_{q \to \infty} \frac{1-\nb^{\mathrm{ns}}(q)}{1-\nb(q)}\geq 1.
$$
In view of inequality~\eqref{ineq:1/e-and-1}, we deduce that
$$
\lim_{q \to \infty} \frac{1-\nb^{\mathrm{ns}}(q)}{1-\nb(q)}=1.
$$
Thus, the event of being smooth and the event of being blocking are asymptotically independent. 
\end{proof}

\begin{rem}\label{rem:smooth-effective-constant}
In our analysis above, the quantities $\nu(q, B)$ and $\nu^{\mathrm{ns}}(q,B)$ already take into account the condition $d\to\infty$. However, one can obtain an effective estimate by more careful computation (especially by analyzing equation~\eqref{eq:mu-ns-q-B} more closely). We could get the same final answer by considering degree $d$ plane curves with $d>q^{3(q^2+q+1)+\epsilon}$ instead of letting $d\to\infty$.
\end{rem}

\begin{rem}
The class of geometrically irreducible plane curves includes smooth curves. Moreover, it is known that the density of geometrically irreducible curves of degree $d$ over a fixed $\F_q$ approaches $1$ as $d\to\infty$ \cite{P04}*{Proposition 2.7}. The same result as in Theorem~\ref{thm:smooth-blocking-independence} will hold, that is, being blocking and geometrically irreducible are asymptotically independent. One can combine the proof idea in our result with the analysis in Poonen~\cite{P04}*{Proposition 2.7} to obtain a better effective version for the statement involving independence. For example, the bound $d>q^{\varepsilon}$ would suffice instead of $d>q^{3(q^2+q+1)+\epsilon}$ as in Remark~\ref{rem:smooth-effective-constant}.
\end{rem}

\begin{rem}
Our proof can be adapted to show the asymptotic independence for the events of being smooth and satisfying other combinatorial properties described by passing through a specific collection of points. For example, if $k$ is fixed, we can show that being a $k$-arc and being smooth are independent events. The reason is that the inequality~\eqref{ineq:1/e-and-1} holds when any other property replaces the condition of being blocking. Moreover, the upper bound in Theorem~\ref{thm:lower-upper-bounds} was the only ingredient we needed to ensure that the final limit is $1$. Any other combinatorial structure satisfying an analogous lower bound will yield the same limit.

However, without an analogue of Theorem~\ref{thm:lower-upper-bounds}, we cannot guarantee that the final limit is $1$. In particular, consider the event $E$ that a random plane curve is space-filling (that is, passing through all $q^2+q+1$ distinct $\F_q$-points of $\bP^2$). Then $E$ is not independent of the event that a random plane curve is smooth, and the limiting ratio of the respective probabilities is $1/e$.
\end{rem}

\section{Proportion of blocking curves over small finite fields}\label{sect:proportion-small-fields}

In this section, we have used SageMath \cite{sagemath} to carry out the exact computation for $\nb(q)$ for $q=2, 3, 4$ by applying the explicit formula given in Proposition~\ref{proportion-of-curves}. The formula involves counting all possible frequencies of the number of $\F_q$-rational points contained in a union of $k$ distinct $\F_q$-lines for $k=1, 2, \ldots, q^2+q+1$. The formula directly provides an algorithm of running time $O(q^2 2^{q^2+q+1})$, which is infeasible for $q$ moderately large.

Recall that $\nb(q)$ is the limiting proportion (as $d\to\infty)$ of degree $d$ non-blocking plane curves among all degree $d$ plane curves defined over $\F_q$. As we saw earlier in Section~\ref{sect:proportion-general-case}, the constant $\nb(q)$ coincides with the density of fixed degree $d$ non-blocking curves whenever $d\geq 2q-1$. 

\begin{ex} Let us calculate $\nb(2)$. We have $q^2+q+1=7$ distinct $\F_2$-lines. Using SageMath, we have tabulated the number of points in the union of $k$ lines for $1\leq k\leq 7$.

\begin{center}
\begin{tabular}{ |c|c|c|||c|c|c|} 
 \hline
 Value of $k$ & \# of $\F_2$-points & frequency & Value of $k$ & \# of $\F_2$-points & frequency  \\ 
 \hline
 $k=1$ & $3$ & $7$ & $k=4$ & $7$ & $28$ \\ 
 $k=2$ & $5$ & $21$ & $k=5$ & $7$ & $21$ \\ 
 $k=3$ & $6$ & $28$ &  $k=6$ & $7$ & $7$\\
 $k=3$ & $7$ & $7$ &  $k=7$ & $7$ & $1$ \\
 $k=4$ & $6$ & $7$ & & & \\
 \hline
\end{tabular}
\end{center}
Applying Proposition~\ref{proportion-of-curves}, we get that for $d\geq 3$,
\begin{align*}
    \mu_d(\mathcal{E}) &= 7\cdot\left(\frac{1}{2}\right)^3 - 21\cdot\left(\frac{1}{2}\right)^5 + 28 \cdot\left(\frac{1}{2}\right)^6 + 7\cdot \left(\frac{1}{2}\right)^7 \\
    &-7\cdot\left(\frac{1}{2}\right)^6-28\cdot\left(\frac{1}{2}\right)^7 + (21-7+1)\cdot\left(\frac{1}{2}\right)^7 \\ 
    &= 7\cdot\left(\frac{1}{2}\right)^3 - 21\cdot\left(\frac{1}{2}\right)^5 + 21\cdot \left(\frac{1}{2}\right)^6 - 6\cdot \left(\frac{1}{2}\right)^7 = \frac{1}{2}
\end{align*}
We conclude that for each $d\geq 3$, the density of curves of degree $d$ over $\F_2$ that do not provide a blocking set is equal to $1/2$. In particular, $\nb(2)=1/2$. 

In contrast, the density is different when $d=1$ or $d=2$. For $d=1$, we know that $100\%$ of the curves are blocking. For $d=2$, the only blocking conics are unions of (not necessarily distinct) $\F_q$ lines. Thus, the proportion of conic blocking curves is given by:
$$
\frac{(\frac{1}{2}(q^2+q+1)(q^2+q)+(q^2+q+1))(q-1)+1}{q^6}
$$
which evaluates to $11/32$ when $q=2$. Thus, the proportion of non-blocking conics over $\F_2$ is $\frac{21}{32}$.
\end{ex}

\begin{ex} Let us calculate $\nb(3)$. We have $q^2+q+1=13$ distinct $\F_3$-lines. Using SageMath, we have tabulated the number of points in the union of $k$ lines for $1\leq k\leq 13$.

\begin{center}
\begin{tabular}{ |c|c|c||c|c|c| } 
 \hline
 Value of $k$ & \# of $\F_3$-points & frequency & 
  Value of $k$ & \# of $\F_3$-points & frequency \\ \hline
 $k=1$ & $4$ & $13$ &  $k=6$ & $13$ & $702$ \\
 $k=2$ & $7$ & $78$ &  $k=7$ & $12$ & $468$ \\
 $k=3$ & $9$ & $234$ &  $k=7$ & $13$ & $1248$ \\
 $k=3$ & $10$ & $52$ &  $k=8$ & $12$ & $117$ \\
 $k=4$ & $10$ & $234$ &  $k=8$ & $13$ & $1170$ \\
 $k=4$ & $11$ & $468$ &  $k=9$ & $12$ & $13$ \\
 $k=4$ & $13$ & $13$ & $k=9$ & $13$ & $702$ \\
 $k=5$ & $11$ & $468$ &  $k=10$ & $13$ & $286$ \\
 $k=5$ & $12$ & $702$ &  $k=11$ & $13$ & $78$ \\
 $k=5$ & $13$ & $117$ &  $k=12$ & $13$ & $13$ \\
 $k=6$ & $11$ & $78$ &  $k=13$ & $13$ & $1$ \\
 $k=6$ & $12$ & $936$ & & & \\
 \hline
\end{tabular}
\end{center}
Applying Proposition~\ref{proportion-of-curves}, we get that for $d\geq 5$,
\begin{align*}
    \mu_d(\mathcal{E}) &= 13\cdot\left(\frac{2}{3}\right)^4 - 78\cdot\left(\frac{2}{3}\right)^7 + 9 \cdot\left(\frac{2}{3}\right)^9 + (52-234) \cdot\left(\frac{2}{3}\right)^{10} + \\
    & + (-468+468-78)\cdot \left(\frac{2}{3}\right)^{11} + (702-936+468-117+13)\left(\frac{2}{3}\right)^{12}  \\  &+ (-13+117-702+1248-1170+702-286+78-13+1)\left(\frac{2}{3}\right)^{13} \\
    &= \frac{1336688}{1594323} = \frac{2^4\cdot 19\cdot 4397}{3^{13}} \approx 0.8384
\end{align*}
We conclude that for each $d\geq 5$, the density of curves of degree $d$ over $\F_3$ that do not provide a blocking set is approximately equal to $0.8384$. In particular, $\nb(3)\approx 0.8384$. 
\end{ex}

\begin{ex} The proportion of non-blocking curves over $\F_4$ is
$$
\nb(4)= \frac{2112952233969}{2199023255552} = \frac{3^6\cdot 29\cdot 67\cdot 1491727}{2^{41}}\approx 0.96086. 
$$

The table below shows how many $\F_4$-points are contained in a union of $k$ distinct $\F_4$-lines in $\mathbb{P}^2$. For example, when $k=4$, then there are $2520$ configurations of $4$ lines which contain $14$ distinct $\F_4$-points, $3360$ configurations of $4$ lines which contain $15$ distinct $\F_4$-points, and $105$ configurations which contain $17$ distinct $\F_4$-points.

\begin{center}
\begin{tabular}{ |c|c|c||c|c|c| } 
 \hline
 Value of $k$ & \# of $\F_4$-points & frequency & Value of $k$ & \# of $\F_4$-points & frequency \\ 
 \hline
 $k=1$ & $5$ & $21$ & $k=9$ & $18$ & $1120$ \\ 
 $k=2$ & $9$ & $210$ & $k=9$ & $19$ & $42840$ \\ 
 $k=3$ & $12$ & $1120$ & $k=9$ & $20$ & $151200$ \\
 $k=3$ & $13$ & $210$ &  $k=9$ & $21$ & $98770$ \\
 $k=4$ & $14$ & $2520$ &  $k=10$ & $19$ & $13860$ \\
 $k=4$ & $15$ & $3360$ & $k=10$ & $20$ & $140448$ \\
 $k=4$ & $17$ & $105$ & $k=10$ & $21$ & $198408$ \\
 $k=5$ & $15$ & $1008$ & $k=11$ & $19$ & $2520$ \\
 $k=5$ & $16$ & $10080$ & $k=11$ & $20$ & $86688$ \\
 $k=5$ & $17$ & $7560$ & $k=11$ & $21$ & $263508$ \\
 $k=5$ & $18$ & $1680$ & $k=12$ & $19$ & $210$ \\
 $k=5$ & $21$ & $21$ & $k=12$ & $20$ & $37800$ \\
 $k=6$ & $15$ & $168$ & $k=12$ & $21$ & $255920$ \\
 $k=6$ & $17$ & $18480$ & $k=13$ & $20$ & $11760$ \\
 $k=6$ & $18$ & $22680$ & $k=13$ & $21$ & $191730$  \\
 $k=6$ & $19$ & $12600$ & $k=14$ & $20$ & $2520$ \\
 $k=6$ & $21$ & $336$ & $k=14$ & $21$ & $113760$ \\
 $k=7$ & $17$ & $2520$ & $k=15$ & $20$ & $336$ \\
 $k=7$ & $18$ & $31920$ & $k=15$ & $21$ & $53928$  \\
 $k=7$ & $19$ & $55440$ & $k=16$ & $20$ & $21$  \\
 $k=7$ & $20$ & $23520$ & $k=16$ & $21$ & $20328$\\
 $k=7$ & $21$ & $2880$ & $k=17$ & $21$ & $5985$ \\
 $k=8$ & $18$ & $10290$ & $k=18$ & $21$ & $1330$ \\
 $k=8$ & $19$ & $73080$ & $k=19$ & $21$ & $210$ \\
 $k=8$ & $20$ & $93240$ & $k=20$ & $21$ & $21$ \\
 $k=8$ & $21$ & $26880$ & $k=21$ & $21$ & $1$ \\
 \hline
\end{tabular}
\end{center}
\end{ex}

\section*{Acknowledgements}
The authors thank Matilde Lal\'{\i}n and Greg Martin for helpful discussions. The second author is supported by an NSERC Discovery grant. The authors are also grateful for two anonymous referees for their valuable comments.

\bibliographystyle{alpha}
\bibliography{biblio.bib}

\end{document}